\documentclass[12pt]{amsart}
\usepackage{amsmath}
\usepackage{amssymb,amscd}
\usepackage[colorlinks=true, urlcolor=blue,bookmarks=true,bookmarksopen=true, citecolor=blue,hypertex]{hyperref}

\numberwithin{equation}{section}

\newtheorem{defn}{Definition}[section]
\newtheorem{theorem}{Theorem}[section]

\newtheorem{corollary}[theorem]{Corollary}

\newtheorem{lemma}[theorem]{Lemma}

\newtheorem{remark}[theorem]{Remark}

\def \begineq{\begin{equation}}
\def \endeq{\end{equation}}

\def \bb{\mathbb}

\def \CC{{\bb{C}}}

\def \QQ{{\bb{Q}}}
\def \RR{{\bb{R}}}

\def \ZZ{{\bb{Z}}}

\def \({\left(}
\def \){\right)}
\def \<{\langle}
\def \>{\rangle}
\def \bar{\overline}

\begin{document}

\title[Blowdowns on
 four dimensional quasitoric orbifolds]{Blowdowns and McKay correspondence
  on four dimensional quasitoric orbifolds }

\author[S. Ganguli]{Saibal Ganguli}

\address{Stat-Math Unit, Indian
Statistical Institute, Kolkata, India, and Departamento de
Matem\'aticas, Universidad de los Andes, Bogota, Colombia }

\email{saibalgan@gmail.com}

\author[M. Poddar]{Mainak Poddar}

\address{Stat-Math Unit, Indian
Statistical Institute, Kolkata, India, and Departamento de
Matem\'aticas, Universidad de los Andes, Bogota, Colombia}

\email{mainakp@gmail.com}



\abstract
 We prove the existence of torus invariant almost complex
 structure on any positively omnioriented four dimensional
 primitive quasitoric orbifold. We construct pseudo-holomorphic
  blowdown maps for such orbifolds. We prove a version of McKay
  correspondence when the blowdowns are crepant.
 \endabstract

\maketitle

 \section{Introduction}\label{intro}

 Quasitoric orbifolds are generalizations or topological
 counterparts of simplicial projective toric varieties.
 They admit an action of the real torus of half
dimension such that the orbit space has the
 combinatorial type of a simple convex polytope.
  Davis and Januskiewicz
 \cite{[DJ]}, who introduced the notion of quasitoric space,
  showed that the formula for the cohomology ring of a quasitoric
  manifold, and hence of any nonsingular
  projective toric variety, may be deduced by purely
 algebraic topology methods.
  This was generalized to the orbifold case in
 \cite{[PS]}.

  In general quasitoric manifolds do not have integrable or almost
 complex structure. However, they always have stable almost complex structure.
 Moreover, positively omnioriented quasitoric manifolds
  have been known to have
 an almost complex structure, see \cite{[BP]}.
 It was recently proved by Kustarev \cite{[Kus], [Kus2]} that any positively
 omnioriented quasitoric manifold has an almost complex structure
which is torus invariant.
   We extend his result to four dimensional
 primitive quasitoric orbifolds, see theorem
 \ref{acs}. Note that for higher dimensional positively omnioriented
 quasitoric orbifolds the
existence of an almost complex structure, torus invariant or otherwise, remains
 an open problem. We hope to address this in future.

  Inspired by birational geometry of toric varieties, we introduce the
 notion of  blowdown into the realm of quasitoric orbifolds. Our blowdown maps
contract an embedded orbifold sphere (exceptional sphere) to a point. They
 admit a description in terms of a finite collection of coordinate charts, very much in
 the spirit of toric geometry. These maps are torus invariant, continuous and
 diffeomorphism of orbifolds away from the exceptional sphere, see theorem \ref{bldn}.
However they are not morphisms
of orbifolds near the  exceptional sphere. For a blowdown map between two
primitive positively omnioriented
quasitoric orbifolds of dimension four (see remark \ref{primi} and subsection \ref{omnio} for
definitions), we can choose almost complex structures
on them so that the blowdown is an analytic map  near the exceptional sphere
 and an almost complex morphism of orbifolds away from it, see theorem \ref{holo}.
 We explain in corollary \ref{smbd} that the blowdown is a pseudo-holomorphic map
 in the sense that it pulls back invariant pseudo-holomorphic
 functions to similar functions.

 We restrict our study to the case where the exceptional sphere has at most one
 orbifold singularity. Any singularity on a $4$-dimensional primitive positively
 omnioriented quasitoric orbifold may be resolved by a sequence of blowdowns of
 this type.  However as our method for studying pseudo-holomorphicity of blowdowns is
 not very amenable to induction, we leave the the study of such a sequence of
  blowdowns for future work.

     An immediate consequence
   of the existence of an almost complex structure is that Chen-Ruan
  orbifold cohomology \cite{[CR]} is defined.  We  define  when an
 almost complex blowdown is crepant, see definition \ref{defcrep}.
 We prove that the Betti
 numbers of  orbifold cohomology are preserved under a
 crepant blowdown. This is a form of McKay correspondence.
 Such correspondence has been widely studied for
 algebraic orbifolds.

 Masuda \cite{[Mas]} proved the existence of invariant almost complex structure
 on positively omnioriented four dimensional quasitoric manifolds,
 many of which are not toric varieties.
 This shows that our results are not redundant. We give  explicit examples in the
 last section. In fact, using  blowdown we can construct almost complex quasitoric
 orbifolds that are not toric varieties.

 We refer the reader to \cite{[ALR]} for relevant definitions and results regarding orbifolds.

\section{Preliminaries}\label{smooth}

 Many results of this section are part of folklore. However to set up our
notation and due to lack of a good reference, we give an explicit description.

 \subsection{Definition} Let $N$ be a free $\ZZ$-module of rank $2$. Let
  $ T_N := (N \otimes_{\ZZ} \RR) / N \cong \RR^2/ N $
 be the corresponding $2$-dimensional torus. A primitive vector in $N$, modulo sign,
  corresponds to a circle subgroup
 of $T_N$. More generally, suppose $M$ is a free submodule of $N$ of rank $m$. Then
 $T_M := (M \otimes_{\ZZ} \RR) /M $ is a torus of dimension $m$.
  Moreover there is a natural homomorphism of Lie
 groups $\xi_M: T_M \to T_N$ with finite kernel, induced by the inclusion $M \hookrightarrow N$.

\begin{defn}\label{tlambda} Define  $T(M)$ to be the image
  of $T_M$ under $\xi_M$.
 If $M$ is generated  by a vector $\lambda \in N$, denote
   $T(M)$ by $T(\lambda)$.
\end{defn}

\begin{defn} \cite{[PS]} A $4$-dimensional quasitoric orbifold $\bf Y$ is an orbifold whose underlying
  topological space $Y$ has a $T_N$ action, where $N$ is a fixed free
 $\ZZ$-module of rank $2$, such that the orbit space is (diffeomorphic to)
  a $2$-dimensional polytope $P$. Denote the projection map from $Y$ to $P$
  by $\pi: Y \to P$.
  Furthermore every point $x \in Y$ has
\begin{itemize}
\item[A1)]  a $T_N$-invariant neighborhood $ V$, \item[A2)] an
associated free  $\ZZ$-module  $M$ of rank $2$ with an
  isomorphism $\theta: T_M \to U(1)^2$ and an injective module homomorphism $i: M \to N$
  which induces a surjective covering homomorphism $\xi_M : T_M \to T_N $,
\item[A3)]   an orbifold chart
   $(\widetilde{V}, G, \xi)$ over $V$, where $\widetilde{V}$ is $\theta$-equivariantly diffeomorphic to an
    open set in $\CC^2$,  $G = {\ker} \xi_M $ and $\xi: \widetilde{V} \to V$ is an equivariant map i.e.
   $\xi(t\cdot y)= \xi_M(t)\cdot \xi(y)$ inducing a homeomorphism between $\widetilde{V}/G$ and $V$.
\end{itemize}
\end{defn}

 Fix $N$. Let $P$ be a convex polytope in $\RR^2$ with edges $E_i$, $i \in
 \mathcal{E}:= \{1,2,\ldots, m\}$. Identify the set of
 edges  with $\mathcal{E}$.
 A function $\Lambda: \mathcal{E} \to N $ is called a characteristic function for $P$ if
 $ \Lambda(i)$ and $\Lambda(j)$ are linearly independent whenever $E_i$ and $E_j$ meet at a vertex
 of $P$. We write $\lambda_i$ for $\Lambda(i)$ and call it a characteristic vector.

\begin{remark}\label{primi} In this article we  assume that all characteristic vectors are primitive.
Corresponding quasitoric orbifolds have been termed primitive quasitoric orbifold in \cite{[PS]}.
They are characterized by the codimension of singular set being greater than or equal to four.
\end{remark}

 Let $\Lambda$ be a characteristic function for $P$.
 For any face $F$ of $P$, let $N(F)$ be the submodule of $N$ generated by
 $\{ \lambda_i: F \subset E_i \} $. For any point $p \in P$, denote by $F(p)$
 the face of  $P$ whose relative interior contains $p$.
  Define an equivalence relation $\sim$ on the space
 $P \times T_N$ by
 \begin{equation} \label{defequi}
 (p,t) \sim (q,s) \; {\rm if \; and\; only\; if\;} p=q \; {\rm and} \; s^{-1}t \in T(N({F(p)}))
 \end{equation}

 The quotient space $X :=P \times T_N/ \sim$ can be given the
structure of a $4$-dimensional
  quasitoric orbifold. Moreover any  $4$-dimensional primitive quasitoric orbifold may be obtained in
  this way, see \cite{[PS]}. We refer to the pair $(P,\Lambda)$ as a model for the
  quasitoric orbifold.

   The space $X$ inherits an action of $T_N$ with orbit space
   $P$ from the natural action
  on $P \times T_N$. Let $\pi: X \to P$ be the associated quotient map.
   The space $X$ is a manifold if the characteristic vectors $\lambda_i$ and $\lambda_j$
  form a ${\ZZ}$-basis of $N$ whenever the edges $E_i$ and $E_j$ meet at a vertex.

   The points $\pi^{-1}(v) \in X$, where $v$ is any vertex of $P$, are fixed by the
   action of $T_N$. For simplicity we will denote the point $\pi^{-1}(v)$  by $v$ when there
   is no confusion.

 \subsection{Differentiable structure}
 Consider open
 neighborhoods $U_v \subset P$ of the vertices $v$  such that
  $U_v$ is the complement in $P$ of all edges
 that do not contain $v$.
  Let
 \begin{equation}
X_v := \pi^{-1}(U_v) =  U_v \times T_N / \sim
\end{equation}
 For a face $F$ of $P$ containing $v$ there is a natural inclusion of $N(F)$ in $N(v)$.
  It induces an injective homomorphism $T_{N(F)} \to T_{N(v)}$ since a basis of $N(F)$ extends
  to a basis of $N(v)$. We will regard $T_{N(F)}$ as a
  subgroup of $T_{N(v)}$ without confusion.

 Define an equivalence relation $\sim_v$ on $U_v \times T_{N(v)}$ by $(p,t)\sim_v (q,s)$ if
 $p=q$ and $s^{-1}t \in T_{N(F)}$ where $F$ is the face whose relative interior contains $p$.
 Then the space
\begin{equation}
 \widetilde{X}_v:= U_v \times T_{N(v)}/ \sim_v
\end{equation}
is $\theta$-equivariantly diffeomorphic to an open ball in $\CC^2$, where
   $\theta: T_{N(v)} \to U(1)^2$ is an isomorphism, see \cite{[DJ]}. This is also
   evident from the discussion on local models below.

  The map $\xi_{N(v)} : T_{N(v)} \to T_N$ induces a map $\xi_v: \widetilde{X}_v \to X_v$ defined by
   $\xi_v([(p,t)]^{\sim_v}) = [(p,\xi_{N(v)}(t)) ]^{\sim}$ on equivalence classes.
    The kernel of $\xi_{N(v)}$,  $G_v = N/N(v)$, is a finite
  subgroup of $T_{N(v)}$ and therefore has a natural smooth, free action on $T_{N(v)}$
  induced by the group operation.  This induces smooth action of $G_v$ on
  $\widetilde{X}_v$. This action is not free in general. Since $T_N \cong T_{N(v)}/G_v $,  $X_v$
 is homeomorphic to the quotient space $\widetilde{X}_v/G_v$.  An orbifold chart
 (or uniformizing system) on $X_v$ is
  given by $(\widetilde{X}_v, G_v, \xi_v)$.

  Up to homeomorphism we may regard the set $U_v$ as a cone $\sigma(v)$ with the
   same edges as $U_v$.
   The neighborhood $X_v$ is then homeomorphic to $\sigma(v) \times T_N/\sim$.
   We say that a local model for $X$ near $v$ consists of a cone $\sigma$
  and characteristic vectors, say, $\lambda_1$, $\lambda_2$ along its edges
   $E_1$  and $E_2$.

  Let $p_1, p_2$ denote the standard coordinates on  $\RR^2 \supset P$.
  Let $q_1,q_2$ be the coordinates on $N \otimes \RR$ with respect to the standard
  basis of $N$. They correspond to standard angular coordinates on $T_N$.
   The local model where $\sigma = \RR^2_{\ge}:= \{(p_1,p_2) \in \RR^2: p_i \ge 0 \} $ and
    the characteristic vectors are
    $(1,0)$ along $p_1=0$ and $(0,1)$ along $p_2=0$ is called standard.
     In this case  there is a homeomorphism from the $\RR^2_{\ge}
     \times T_N/\sim $ to $\CC^2 = \RR^4$
    given by
    \begin{equation}
    x_i= \sqrt{p_i} \cos (2 \pi q_i), \;\;
    y_i = \sqrt{p_i} \sin (2 \pi q_i) \;\; {\rm where} \; i=1,2.
    \end{equation}

     For any cone $\sigma(v) \subset \RR^2$ with arbitrary characteristic
     vectors we will define a canonical homeomorphism
      $\phi(v): \widetilde{X}_v  \to \RR^4$ as follows.
     Order the edges $E_1, E_2$ of $\sigma(v)$ so that the clockwise angle from $E_1$ to
     $E_2$ is less than $180^{\circ}$.
     Denote the coordinates of the vertex  $v$ by  $(\alpha, \beta)$.
     Let the equations of the edge $E_i$ be $a_i(p_1 - \alpha) + b_i(p_2 - \beta)=0$.
     Assume that the interior of $\sigma(v)$ is contained in the half-plane
     $a_i(p_1 - \alpha) + b_i(p_2 - \beta) \ge 0$.
     Suppose $\lambda_1 =(c_{11}, c_{21})$ and $\lambda_2=(c_{12}, c_{22})$ be the
     characteristic vectors assigned to $E_1$ and $E_2$ respectively.
     If $q_1(v), q_2(v)$ are angular coordinates of an element of $T_N$ with respect
     to the basis $\lambda_1, \lambda_2$ of $N \otimes \RR$, then  the
     standard coordinates $q_1,q_2$ may be expressed as
     \begin{equation}\label{cocoor}
     \left(  \begin{array}{c}  q_1 \\ q_2  \end{array}\right) =
     \left[ \begin{array}{cc} c_{11} & c_{12} \\ c_{21} & c_{22} \end{array} \right]
     \left( \begin{array}{c} q_1(v) \\ q_2(v)  \end{array} \right).
     \end{equation}

     Then define the homeomorphism $\phi(v): \widetilde{X}_v \to \RR^4$ by
     \begin{equation}\label{homeo}
      x_i = x_i(v):= \sqrt{p_i (v)} \cos(2 \pi q_i(v) ), \quad
      y_i = y_i(v):= \sqrt{p_i(v) } \sin( 2 \pi q_i(v) ) \quad {\rm for}\; i=1,2
     \end{equation}
     where
     \begin{equation}\label{homeod}
      p_i (v) =  a_i(p_1-\alpha) + b_i(p_2 - \beta), \quad i=1,2.
       \end{equation}
   Similar homeomorphism has been used in \cite{[Sym]}.

   Now consider the action of $G_v = N/N(v)$ on $\widetilde{X}_v$. An element
     of $G_v$ is represented by a vector $g =a_1 \lambda_1 + a_2 \lambda_2$ in
    $N$ where $a_1, a_2 \in  \QQ$.  The action of $g$ transforms the coordinates
    $(q_1(v), q_2(v))$ to $(q_1(v) + a_1, q_2(v) + a_2)$. If we write
    $z_i = x_i + \sqrt{-1} y_i$, then
    \begin{equation}\label{action}
     g\cdot (z_1, z_2) = (e^{2\pi \sqrt{-1}a_2} z_1, e^{2\pi \sqrt{-1}a_1} z_2)
     \end{equation}
Since $\lambda_1$ and $\lambda_2$ are both primitive, neither of $a_1, a_2$ is an
integer. Therefore the only orbifold singularity on $X_v$ is at the point with coordinates
 $z_1= z_2=0$, namely the vertex $v$.

We show the compatibility of the charts $(\widetilde{X}_v, G_v, \xi_v) $. Let $v_1$ and $v_2$
be two adjacent vertices. Assume that edges $E_1, E_2$  meet at $v_1$ and
edges $E_2, E_3$  meet at $v_2$. Let $\lambda_i $ be the characteristic
vector corresponding to $E_i$. Since all characteristic vectors are primitive, we may assume
without loss of generality that $\lambda_1 = (1,0)$. Suppose $\lambda_2 = (a,b)$ and
$\lambda_3 = (c,d)$. Let $\Delta= ad-bc$.

Up to choice of coordinates we may assume that the edge $E_1$ has equation $p_1=0$,
the edge $E_3$ has equation $p_2 = 0$, and the edge $E_2$ has equation
${\widehat p} =0$. Here $\widehat{p} = p_2 + s p_1 - t$ where
$s $ and $t$ are positive reals.
We assume that the quantities $p_1, p_2 $ and $\widehat{p}$
 are positive in the the interior of the polytope.

 We shall write down explicit coordinates on $\widetilde{X}_{v_i}$.  For this purpose it
    is convenient to express all angular coordinates in terms of $(q_1,q_2)$ by inverting
    equation \eqref{cocoor}.
    \begin{equation}\label{cocoor2}
     \left(  \begin{array}{c} q_1(v) \\ q_2(v)  \end{array}\right) =
     \left[ \begin{array}{cc} d_{11} & d_{12} \\ d_{21} & d_{22} \end{array} \right]
     \left( \begin{array}{c} q_1 \\ q_2 \end{array} \right)
     \end{equation}
     where the matrix $(d_{ij})$ is the inverse of the matrix $(c_{ij})$. Then by
 equations \eqref{homeo}, \eqref{homeod}, \eqref{cocoor2} we have the
following expressions for coordinates $z_i(v_j):= x_i(v_j) + \sqrt{-1}y_i(v_j)$
on $\widetilde{X}_{v_j}$.

\begin{equation}
 \begin{array}{ll}
       z_1(v_1) := \sqrt{p_1} e^{2\pi \sqrt{-1} (q_1 -\frac{a}{b}q_2)}, &
    z_2(v_1) := \sqrt{\widehat{p}} e^{2\pi \sqrt{-1} \frac{1}{b}q_2} \\
   \\
   z_1(v_2) := \sqrt{\widehat{p}}  e^{2\pi \sqrt{-1} \frac{dq_1- c q_2}{\Delta}}, &
   z_2 (v_2)  := \sqrt{p_2} e^{2\pi \sqrt{-1} ( \frac{-bq_1 + a q_2}{\Delta})}
\end{array}
\end{equation}

The coordinates on $\widetilde{X}_{v_2}$ are related to those on $\widetilde{X}_{v_1}$
over the intersection $X_{v_1} \bigcap X_{v_2}$ as follows.
\begin{equation}\label{trans}
\begin{array}{ll}
 z_1(v_2)= z_1(v_1)^{\frac{d}{\Delta}} z_2(v_1) \sqrt{p_1}^{\frac{-d}{\Delta} },
&  z_2(v_2) = z_1(v_1)^{\frac{-b}{\Delta}} \sqrt{p_2} \sqrt{p_1}^{\frac{b}{\Delta}}
\end{array}
\end{equation}

 Take any point $x \in X_{v_1} \bigcap X_{v_2} $. Let $\widetilde{x}$ be a
preimage of $x$ with respect to $\xi_{v_1}$. Choose a small ball $B(\widetilde{x},r)$
around  $\widetilde{x}$ such that $(g\cdot B(\widetilde{x},r)) \bigcap B(\widetilde{x},r)$
is empty for all nontrivial $g \in G_{v_1}$. Then $( B(\widetilde{x},r), \{1\},\xi_{v_1})$
is an orbifold chart around $x$. This chart admits natural embedding (or injection) into
the chart $(\widetilde{X}_{v_1}, G_{v_1}, \xi_{v_1})$ given by inclusion. We show that
for sufficiently small value of $r$, this chart embeds into
$(\widetilde{X}_{v_2}, G_{v_2}, \xi_{v_2})$ as well. Choose a branch of
$z_1(v_1)^{\frac{1}{\Delta}}$ so that the branch cut does not intersect
$ B(\widetilde{x},r)$. Assume $r$ to be small enough so that the functions
$z_1(v_1)^{\frac{d}{\Delta}}$ and $z_1(v_1)^{\frac{b}{\Delta}}$ are one-to-one
on $ B(\widetilde{x},r)$. Then  equation \eqref{trans} defines a smooth
embedding $\psi$ of $ B(\widetilde{x},r)$ into $\widetilde{X}_{v_2}$. Note that $p_1$
and $p_2$ are smooth nonvanishing functions on
 $\xi_{v_1}^{-1} ( X_{v_1} \bigcap X_{v_2})$ .  Assume $r$ to be small enough so that
$(h\cdot \psi(B(\widetilde{x},r))) \bigcap \psi( B(\widetilde{x},r))$ is empty for all $h$ in
 $G_{v_2}$.  Then $(\psi, id):( B(\widetilde{x},r), \{1\},\xi_{v_1}) \to
  (\widetilde{X}_{v_2}, G_{v_2}, \xi_{v_2}) $ is an embedding of orbifold charts.

 \begin{remark}\label{rembold}  We will denote the topological space $X$ endowed with
the above orbifold structure by ${\bf X}$. In general we denote the underlying space of
an orbifold ${\bf Y}$ by $Y$. We denote the set of smooth points of $Y$, i.e. points having
trivial local group, by $Y_{reg}$.
\end{remark}

\begin{remark}\label{remstd}
The equivariant homeomorphism or diffeomorphism type
  of a quasitoric orbifold does not depend on the choice of signs of the characteristic
   vectors. However these signs do effect the
 local complex structure on the coordinate charts obtained via the
  pullback of the standard complex
 structure on $\CC^2/G_v$. A theorem of Prill \cite{[Pri]} proves that the analytic
 germ of singularity at $v$ is characterized by the linearized action of $G_v$.
 \end{remark}

     The following lemma shows that the orbifold structure on ${\bf X}$ does not depend on the
      shape of the polytope $P$.

    \begin{lemma}\label{clas} Suppose ${\bf X}$ and ${\bf Y}$ are four dimensional quasitoric
    orbifolds whose orbit spaces
    $P$ and $Q$ are diffeomorphic and the characteristic vector of any edge of $P$
    matches with the characteristic vector of the corresponding edge of $Q$. Then ${\bf X}$
    and ${\bf Y}$ are equivariantly diffeomorphic.
    \end{lemma}

 \begin{proof} Pick any vertex $v$ of $P$.  For simplicity we will write $p_i$ for $p_i(v)$, and
$q_i$ for $q_i(v)$. Suppose the diffeomorphism
$f : P_1 \to P_2$ is given near $v$ by $f(p_1,p_2)= (f_1,f_2)$. It induces a map of
 local charts  $\widetilde{X}_v \to \widetilde{Y}_{f(v)} $ by
\begin{equation}\label{map}
(\sqrt{p_i} \cos(q_i), \sqrt{p_i} \sin(q_i)) \mapsto (\sqrt{f_i} \cos(q_i), \sqrt{f_i} \sin(q_i)) \quad
{\rm for}\; i= 1,2.
\end{equation}
This is a smooth map if the functions $\sqrt{f_i/p_i}$ are smooth functions of $p_1,p_2$. Without
loss of generality let us consider the case of $\sqrt{f_1/p_1}$. We may write
\begin{equation}\label{taylor}
f_1(p_1,p_2) = f_1(0,p_2) +  p_1 \frac{\partial f_1}{\partial p_1} (0,p_2) + p_1^2 g(p_1,p_2)
\end{equation}
where $g$ is smooth, see section 8.14 of \cite{[Dieu]}. Note that $f_1(0,p_2)=0$ as $f$ maps
the edge $p_1=0$ to the edge $f_1 =0$. Then it follows from equation \eqref{taylor} that
$f_1/p_1$ is smooth. We have
\begin{equation}\label{taylor2}
\frac{f_1}{p_1} =  \frac{\partial f_1}{\partial p_1} (0,p_2) + p_1 g(p_1,p_2)
\end{equation}
Note that $\frac{f_1}{p_1}$ is nonvanishing away from $p_1=0$.
Moreover we have
\begin{equation}
\frac{f_1}{p_1}=  \frac{\partial f_1}{\partial p_1} (0,p_2) \; \; {\rm when}\; p_1 =0.
\end{equation}
Since $f_1(0,p_2) $ is identically zero, $ \frac{\partial f_1}{\partial p_2} (0,p_2) =0 $.
As the Jacobian of $f$ is nonsingular we must have
\begin{equation}
\frac{\partial f_1}{\partial p_1} (0,p_2) \neq 0
\end{equation}
Thus  $\frac{f_1}{p_1}$ is nonvanishing even when $p_1=0$. Consequently
$\sqrt{f_1/p_1}$ is smooth.
Therefore the map \eqref{map} is smooth and induces an  isomorphism of orbifold charts.
\end{proof}



    \subsection{Torus action} An action of a group $H$ on an orbifold ${\bf Y}$ is an
    action of $H$ on the underlying space $Y$ with some extra conditions. In particular
    for every sufficiently small $H$-stable neighborhood $U$ in $Y$ with uniformizing
    system $(\widetilde{U},G,\xi)$, the action should lift to an action of $H$ on $\widetilde{U}$ that commutes
    with the action of $G$. The $T_N$-action on the underlying topological space of a
    quasitoric orbifold does not lift to an action on the orbifold in general.

    \subsection{Metric} Any cover of $X$ by $T_N$-stable open sets induces an
    open cover of $P$. Choose a smooth partition of unity on the polytope $P$ subordinate
    to this induced cover.
    Composing with  the projection map $\pi: X \to P$ we obtain a partition of
    unity on $X$ subordinate to the given cover, which is $T_N$-invariant. Such a
    partition of unity is smooth as the map
    $\pi$ is smooth, being locally given by maps $p_j = x_j^2 + y_j^2$.

    \begin{defn}
    By a torus invariant metric on ${\bf X}$ we will mean a metric on ${\bf X}$ which is
    $T_{N(v)}$-invariant in some neighborhood of each vertex $v$ and $T_N$-invariant on
     $X_{reg}$. \end{defn}

     For instance, choose a $T_{N(v)}$-invariant metric on each $\widetilde{X}_v$. Then
      using a partition of unity
     as above we can define a metric on ${\bf X}$. Such a metric is $T_N$-invariant on
     $X_{reg}$. We use variants of this construction in what follows.

    \subsection{Characteristic suborbifolds} The $T_N$-invariant subset $\pi^{-1}(E)$, where
    $E$ is any edge of $P$, is a suborbifold of ${\bf X}$. It is called a characteristic
     suborbifold. Topologically it is a sphere. It can have orbifold singularity only at the
     two vertices. If $\lambda$ is the characteristic vector attached to $E$, then
     $\pi^{-1}(E)$ is fixed by the circle subgroup $T(\lambda)$ of $T_N$.
      A characteristic suborbifold is a quasitoric orbifold, see \cite{[PS]}.

     \subsection{Orientation} Consider the manifold case first.
       Note that for any vertex $v$
      $dp_i(v) \wedge dq_i(v) = dx_i(v) \wedge dy_i(v)$. Therefore $\omega(v):=
      dp_1(v)\wedge dp_2(v) \wedge
      dq_1(v) \wedge dq_2(v)$ equals $ dx_1(v) \wedge dx_2(v) \wedge dy_1(v) \wedge dy_2(v) $.
       The standard coordinates $(p_1, p_2)$ are related to $(p_1(v), p_2(v))$
      by a diffeomorphism. Similarly for $(q_1,q_2)$ and $(q_1(v), q_2(v))$. Therefore
      $\omega:= dp_1 \wedge dp_2 \wedge dq_1 \wedge dq_2$ is a nonzero multiple of each
      $\omega(v)$ and defines a nonvanishing form on $X$.
       Thus a choice of orientations for $P \subset \RR^2$
      and $T_N$ induces an orientation for $X$.

       In the orbifold case
      the action of $G_v$ on $\widetilde{X}_v$, see equation \eqref{action},
        preserves $\omega(v)$ for each
      vertex $v$ as $dx_i(v) \wedge dy_i(v)= \frac{\sqrt{-1}}{2} dz_i(v)\wedge d{\bar z}_i(v)$.
       Hence the same conclusion holds.

  \subsection{Omniorientation}\label{omnio} An omniorientation is a choice of orientation for the
  orbifold as well as an orientation for each characteristic suborbifold.
  At any vertex $v$, the $G_v$-representation ${\mathcal T}_0 \widetilde{X}_v$ splits into the direct
  sum of two $G_v$-representations corresponding to the linear subspaces $z_i(v)=0$.
  Thus we have a decomposition of the orbifold tangent space
  ${\mathcal T}_v{\bf X}$ as a direct sum of tangent spaces of the two characteristic
   suborbifolds that meet at $v$.
   Given an omniorientation, we set the
  sign of a vertex $v$ to be positive if the orientations of
   ${\mathcal T}_v({\bf X})$ determined by the
  orientation of ${\bf X}$ and orientations of characteristic suborbifolds coincide. Otherwise
  we say that sign of $v$ is negative. An omniorientation is then said to be positive if
  each vertex has positive sign.

   Note that the normal bundle of any characteristic suborbifold is naturally oriented
   by the action of its isotropy circle. The action and hence the orientation depends on
   the sign of the characteristic vector. Thus
   for a fixed orientation on ${\bf X}$ an omniorientation
    is determined by a choice of sign for each
   characteristic vector. Assume henceforth that
   ${\bf X}$ is oriented via standard orientations on $P$
   and $T_N$.  Then an omniorientation on ${\bf X}$ is positive if
    the matrix of adjacent characteristic vectors, with clockwise
    ordering of adjacent edges,
    has positive determinant at each vertex.

\section{Almost complex structure}\label{thmacs}

 Kustarev \cite{[Kus]} showed that the obstruction to existence of a torus invariant
 almost complex structure on a quasitoric manifold, which is furthermore orthogonal
 with respect to a torus invariant metric, reduces to the obstruction to its existence
 on a section of the orbit map. We use the same principle here for $4$-dimensional
  quasitoric orbifolds. The obstruction theory for orbifolds in higher dimensions
   seems to be more complicated.

  Let ${\bf X}$ be a positively omnioriented $4$-dimensional primitive quasitoric orbifold.

 \begin{defn} We say that an almost complex structure on ${\bf X}$ is
  torus invariant if it is $T_{N(v)}$-invariant in some neighborhood of each vertex $v$
  and $T_N$-invariant on $X_{reg}$.
 \end{defn}

  Denote  the set of all $i$-dimensional faces of $P$ by $sk_{i}(P)$.
  We refer to $\pi^{-1}(sk_{i}(P))$ as the $i$-th skeleton of $X$.
   We fix an embedding
  $\iota : P \longrightarrow X$ that satisfies
   \begin{equation}\label{iota}
  \begin{array}{l}
        \pi\circ \iota = id  \quad {\rm and}\\
     \iota|_{int(G)} \; {\rm is\; smooth \; for\;    any \;    face \; }  G \subset  P.
   \end{array}
   \end{equation}
   A choice of $\iota$ is given by the composition
   $P \stackrel{i}{\rightarrow} P\times T_N \stackrel{j}{\rightarrow} X $ where $i$ is the inclusion
   given by $i(p_1,p_2) = (p_1, p_2, 1,1)$ and $j$ is the quotient map that defines $X$.


    We also fix a torus invariant metric $\mu$ on ${\bf X}$ as follows. Choose an open cover of $P$
 such that each vertex of $P$ has a neighborhood which
 is contained in exactly one open set of the
 cover.
 This induces a cover of $X$. On each
 open set $W_v$  of this cover, corresponding to the vertex $v$,
 choose the standard metric with respect to the coordinates in \eqref{homeo}.
 On the remaining open sets, choose any $T_N$-invariant metric.
    Then use a $T_N$-invariant partition of unity, subordinate to the cover,
to glue these metrics and obtain $\mu$.

 \begin{theorem}\label{acs} There exists a torus invariant almost complex structure
  $J$ on ${\bf X}$ such that $J$ is orthogonal
with respect to $\mu$.
    \end{theorem}

  \begin{proof}  Choose small  orbifold charts  $(\widetilde{X}^{\prime}_v, G_v, \xi_v)$ around
   each vertex $v$ where ${X}^{\prime}_v \subset W_v$.
    Choose coordinates $x_i(v), \,y_i(v),\, i=1,2$ on $\widetilde{X}^{\prime}_v$
     according to \eqref{homeo}. Declare $z_i(v) = x_i(v) + \sqrt{-1} y_i(v) $.
     Choose the standard complex structure
$J_v$ on $\widetilde{X}^{\prime}_v$ with respect to these coordinates,
i.e. $z_1(v), z_2(v)$ are holomorphic coordinates
under $J_v$. Since  $J_v$ commutes with action of $T_{N(v)}$
and $G_v$ is a subgroup of $T_{N(v)}$, we may regard
$J_v$ as a torus invariant complex structure on a neighborhood of $v$
 in the orbifold ${\bf X}$. Note that
these local complex structures are orthogonal with respect to $\mu$ near the
vertices.

 We will first construct an almost complex structure on the first skeleton
 of $X$ that agrees with the above local complex structures.
  Let $E$ be any edge of $P$. Suppose $E$ joins the vertices $u$ and $v$.
 Assume that $\lambda$ is the characteristic vector attached to $E$. The characteristic
 suborbifold corresponding to $E$ is an orbifold sphere which we denote by ${\bf S}^2$.
 Let $\nu$ denote the orbifold normal bundle to ${\bf S}^2$ in ${\bf X}$. This is an
 orbifold line bundle with action of $T(\lambda)$, see definition \ref{tlambda}.

 Take any $x \in {\bf S}^2$.
 Observe, after lemma 3.1 of \cite{[Kus]}, that $\nu_x$ and $\mathcal{T}_x{\bf S}^2$
   are orthogonal with respect to $\mu$. This may be verified as follows. Take any
   nonzero vectors $\eta_1 \in \nu_x$ and $\eta_2 \in \mathcal{T}_x {\bf S}^2$.
   Let $\theta$ be an element of $T(\lambda)$ that acts on $\nu_x$ as multiplication by
   $-1$. Since $T(\lambda)$ acts trivially on $\mathcal{T}_x {\bf S}^2$, we have
   $$\mu(\eta_1, \eta_2)= \mu(\theta \cdot \eta_1, \eta_2 ) = \mu (-\eta_1, \eta_2) = 0. $$
  Thus we may split the construction of an orthogonal $J$ into
 constructions of orthogonal
 almost complex structures on $\nu$ and $\mathcal{T}{\bf S}^2$.

Define the restriction of $J$ to $\nu$ as rotation by the angle
$\frac{\pi}{2}$ with respect to
 the metric $\mu$ in the counterclockwise direction as specified by the orientation
 of $\nu$ obtained from the $T(\lambda)$ action. Then $J|_{\nu}$ is torus invariant as
$\mu$ is preserved by torus action.


 Recall that the space of all
 orthogonal complex structures on the oriented vector space $\RR^{2}$ is parametrized by
 $SO(2)/U(1)$ which is a single point. We orient $\mathcal{T}{\bf S}^2$ according to the
  given omniorientation. Consider the path $\iota(E) \in {\bf S}^2$ given
by the embedding $\iota$ in \eqref{iota}. Since this path is
contractible, the restriction of $\mathcal{T}{\bf S}^2$ on it is
trivial. Thus there is a canonical choice of an orthogonal almost
complex structure on $\mathcal{T}{\bf S}^2|_{\iota(E)}$. We want
this structure to agree with the complex structures
 $J_u|_{\mathcal{T}{\bf S}^2}$ and $J_v|_{\mathcal{T}{\bf S}^2}$
near the vertices, chosen earlier. This is possible since the
omniorientation is positive. Then we use the torus action to
define $J|_{\mathcal{T}{\bf S}^2}$ on each point $x$ in
$\bf{S}^2$. Find $y$ in $\iota(E)$ and $\theta\in T_N/T(\lambda)$
such that $x = \theta \cdot y$. Then define
  $J(x) = d\theta \circ J(y) \circ d\theta^{-1} $.  This completes the construction
of a torus invariant orthogonal $J$ on the $1$-skeleton of ${\bf X}$.

 Choose a simple loop $\gamma$ in
$P$ that goes along the edges for the most part but avoids the vertices.
By the previous step of our construction, $J$ is given on $\iota(\gamma)$.
Let $D$ be the disk in $P$ bounded by $\gamma$. The set
$X_0 := \pi^{-1}(D) \subset X$ is a smooth manifold with boundary.
The restriction of
$\mathcal{T}X_0 $ to $\iota(D)$ is a trivial vector bundle.
 Fix a trivialization.
Recall that the space of all orthogonal complex structures on  oriented vector
 space $\RR^{4}$, up to isomorphism, is homeomorphic to $SO(4)/U(2)$. This
is a simply connected space. Thus $J$ may be extended from $\iota({\gamma})$
to $\iota(D)$. Then we produce a $T_N$-invariant  orthogonal $J$ on $\mathcal{T}X_0 $
by using the $T_N$ action as in the last paragraph. This completes the proof.
 \end{proof}


\section{Blowdown}\label{blow}

  Our  blowdown is  analogous to partial resolution of singularity in complex geometry.
Topologically it is an inverse
 to the operation of connect sum with a complex $2$-dimensional
 weighted projective space. At the combinatorial level,
  it corresponds to deletion of
 an edge and its characteristic vector. To be precise, the polytope is modified by removing one edge
  and extending its neighboring edges till they meet at a new vertex, using lemma
  \ref{clas} if necessary.

   Suppose the orbifold ${\bf X}$ corresponds to the model
 $(P,\Lambda)$. Suppose the edge $E_2$ of $P$ is deleted and
 the neighboring edges $E_1$ and $E_3$ are extended
 produce a new polytope $\widehat{P}$. Denote the characteristic vector
 attached to the edge $E_i$ by $\lambda_i$. Assume that $\lambda_1$ and
$\lambda_3$ are linearly independent. Then $\widehat{P}$ inherits a
 characteristic function $\widehat{\Lambda}$ from $\Lambda$. Let ${\bf Y}$
 be the orbifold corresponding to the model $(\widehat{P}, \widehat{\Lambda})$.
We only consider the case where at least one of the vertices of $E_2$ correspond
to a smooth point of ${\bf X}$.
 Then there exists a continuous $T_N$-invariant map of underlying
 topological spaces $\rho: X \to Y$ called a blowdown. This map contracts
the sphere $\pi^{-1}(E_2) \subset X$ to a point in the image and it is a
diffeomorphism away from this sphere. The sphere $\pi^{-1}(E_2)$ is
called the exceptional set.

 The blowdown
 is not an orbifold morphism near the exceptional set as it cannot
 be lifted locally to a continuous equivariant map on orbifold charts. This is
 not surprising as it
 also happens  for resolution of quotient singularities
  in algebraic geometry.  However we can give a neighborhood of the exceptional
  set an analytic structure
  such that the blowdown is analytic in this neighborhood.
  We can extend these local complex structures on ${\bf X}$ and
   ${\bf Y}$ to almost complex structures so that the blowdown map is an almost
   complex diffeomorphism of orbifolds away from the exceptional set. Moreover the
 blowdown is a $J$-holomorphic map in a natural sense described in corollary \ref{smbd}.

 \begin{theorem}\label{bldn} If $\det [\lambda_1,\lambda_2] =1$ and  $0 < \det [\lambda_2,\lambda_3]
  \le \det [\lambda_1,\lambda_3]$, then there exists a continuous map $\rho:  X \to Y$
   which is a diffeomorphism away from
 the set $\pi^{-1}(E_2) \subset X$.
  \end{theorem}

  \begin{proof}
  Let $v_1$ and $v_2$ be the vertices of $P$ corresponding to $E_1 \bigcap E_2$
  and $E_2 \bigcap E_3$ respectively. Let $w$ be the vertex of $\widehat{P}$ where
  the extended edges $E_1$ and $E_3$ meet.

    Up to change of basis of $N$ we may assume that
   $\lambda_1=(1,0)$, $ \lambda_2=(0,1)$, and  $ \lambda_3=(-k,m)$
   where $0 < k \le m$ are positive integers.
   Then $v_{1}$ is a smooth point in ${\bf X}$, but $v_2$ is a
  possibly singular point  in ${\bf X}$ with local group
    $\ZZ_k$. The point $w$ in ${\bf Y}$ has
    local group $\ZZ_m$.

    Up to choice of coordinates, the equations
  of the sides $E_1$ and $E_3$ may be assumed to be $p_1= 0$ and $p_2 =0$ respectively.
  Suppose the equation of $E_2$ is $\widehat{p} :=p_2 + s p_1 - t = 0$
   where $s $ and $t$ are  positive constants.

   Choose small positive numbers $\epsilon_1 < \epsilon_2< 1$ and a
   non-decreasing function $\delta: [0,\infty) \to \RR$ which is smooth
   away from $0$ such that

   \begin{equation}
   \delta(t)= \left\{ \begin{array}{ll} t^{\frac{1}{m}} & {\rm if} \, t < \epsilon_1 \\
   1 & {\rm if} \, t > \epsilon_2 \end{array} \right.
   \end{equation}

   The  blow down map $\rho: (P\times T_N/\sim) \to (\widehat{P} \times T_N/\sim) $
    may be defined by
   \begin{equation}
   \rho(p_1,p_2, q_1,q_2) = (\delta(\widehat{p})^{k} p_1, \delta(\widehat{p}) p_2, q_1,q_2)
   \end{equation}

It is enough to study the map $\rho$ in the open sets $X_{v_1}$  and $ X_{v_2}$, as it is
identity elsewhere.
The coordinates on  $\widetilde{X}_{v_1}$,  $\widetilde{ X}_{v_2}$ and
$\widetilde{Y}_{w}$ are
\begin{equation}\label{coords}
  \begin{array}{ll}
       z_1(v_1) := \sqrt{p_1} e^{2\pi \sqrt{-1} (q_1)}, &
    z_2(v_1) := \sqrt{\widehat{p}} e^{2\pi \sqrt{-1} q_2} \\
   \\
   z_1(v_2) := \sqrt{\widehat{p}}  e^{2\pi \sqrt{-1} \frac{mq_1+kq_2}{k}}, &
   z_2 (v_2)  := \sqrt{p_2} e^{2\pi \sqrt{-1} ( \frac{-q_1}{k})}\\
    \\
    z_1(w) :=  \sqrt{r_1}    e^{2\pi \sqrt{-1}\frac{mq_1+kq_2}{m}}, &
    z_2(w) :=\sqrt {r_2}  e^{2\pi \sqrt{-1}\frac{q_2}{m}}.
   \end{array}
   \end{equation}
For questions related to smoothness, it is convenient to describe (lift of) $\rho$ in terms of these coordinates.
 The formulas that follow make sense on suitable (small) open sets in the complement of the
exceptional sphere or its image
(and with choice of branch of roots), although
we may sometimes neglect to explicitly say so for the sake of brevity.

  The restriction   $\rho: X_{v_1} \to Y_w $ is given by
   \begin{equation}\label{eqbldn3}  \begin{array}{ll}
   z_1(w) = z_1(v_1) z_2(v_1)^{\frac{k}{m}} \sqrt{\frac{\delta(\widehat{p})^{k}}{\widehat{p}^{\frac{k}{m}}}},
   & z_2(w) = z_2(v_1)^{\frac{1}{m}} \sqrt{\frac{\delta(\widehat{p})p_2 }{\widehat{p}^{\frac{1}{m}}} }.
   \end{array} \end{equation}
The function $p_2$  is nonzero on $X_{v_1}$ whereas the functions  $\delta(\widehat{p})$ and $\widehat{p}$
 and $z_2(v_1)$ are smooth and nonvanishing
away from $\widehat{p}=0$. Hence $\rho$ is a smooth map on $X_{v_1} \bigcap \pi^{-1}(E_2)^c$.
   To show that $\rho$ is  a diffeomorphism away from the exceptional set we exhibit a suitable inverse
map $\rho^{-1}$.
 The map $\rho^{-1}: \rho(X_{v_1}) \bigcap \{(0,0)\}^c \to X_{v_1}$ is given by
\begin{equation} \begin{array}{lll}
   z_1(v_1) = \frac{z_1(w)}{z_2(w)^{k}} \sqrt{\frac{r_2^{k}}{(\delta(\widehat{p}))^{k}}},
   & \quad & z_2(v_1) = z_2(w)^{m} \sqrt{\frac{\widehat{p}}{r_2^{m}}}.
   \end{array}
   \end{equation}
Note that $r_2$ is a smooth function of $x_2(w)$ and $y_2(w)$. Moreover
   $r_2 = p_2 \delta(\widehat{p})$ does not vanish on $\rho(X_{v_1})-\{(0,0)\}$.
   Hence $z_2(w)$ does not vanish on this set either. Same holds for the functions
   $\widehat{p}$ and $\delta(\widehat{p})$.
 It only remains to show that $\widehat{p}$ is a smooth function
   of $x_i(w)$s and $y_j(w)$s.

   The map $f: P \to \widehat{P}$ given by $(r_1, r_2) =
   (\delta(\widehat{p})^k p_1, \delta(\widehat{p}) p_2)$ has Jacobian
   \begin{equation}
   J(f) = \left[ \begin{array}{ll}
           \delta(\widehat{p})^k +
           s k p_1 \delta(\widehat{p})^{k-1} \delta^{\prime}(\widehat{p})
          & k p_1 \delta(\widehat{p})^{k-1}  \delta^{\prime}(\widehat{p}) \\
            & \\
           s p_2 \delta^{\prime}(\widehat{p})
          & \delta(\widehat{p}) +  p_2 \delta^{\prime}(\widehat{p})
          \end{array} \right]
   \end{equation}

   Since $\det(J(f)) = \delta(\widehat{p})^{k+1} +
    \delta(\widehat{p})^k \delta^{\prime}(\widehat{p})(p_2 + s k p_1) > 0 $,
   $f$ is a diffeomorphism away from $\widehat{p} = 0$.
   Thus away from $(r_1, r_2) = (0,0)$, $p_1$ and $p_2$ are smooth
   functions of $(r_1, r_2)$. Hence $\widehat{p}= p_2 + s p_1 -t$ is also a smooth
   function of $r_1$ and $r_2$ away from the origin. Consequently
   $\widehat{p}$ is a smooth function of $x_i(w)$s and $y_j(w)$s
   away from the origin.

     Similarly the blowdown map  restriction   $\rho: X_{v_2} \to Y_w $ given by
\begin{equation}\label{eqbldn4}  \begin{array}{ll}
   z_1(w) = z_1(v_2)^{\frac{k}{m}}\sqrt{\frac{(\delta(\widehat{p}))^{k}p_1 }{\widehat{p}^{\frac{k}{m}}}},
   & z_2(w) = z_1(v_2)^{\frac{1}{m}}z_2(v_2) \sqrt{\frac{\delta(\widehat{p})}{\widehat{p}^{\frac{1}{m}}}},
   \end{array} \end{equation}
   is smooth on $X_{v_2} \bigcap \pi^{-1}(E_2)^c $ as the functions $p_1$, $\widehat{p}$ and $z_1(v_2)$ are
   nonvanishing there.

The map $\rho^{-1}: \rho(X_{v_2}) \bigcap \{(0,0)\}^c \to X_{v_2}$ is given by
\begin{equation} \begin{array}{lll}
   z_1(v_2) = \frac{z_2(w)}{z_1(w)^{1/k}} \sqrt{\frac{r_1^{\frac{1}{k}}}{\delta(\widehat{p})}},
   & \quad & z_1(v_2) = z_1(w)^{\frac{m}{k}} \sqrt{\frac{\widehat{p}}{r_1^{\frac{m}{k}}}}
   \end{array} \end{equation}
The diffeomorphism argument is same as in the case of the vertex
$v_1$. This completes the proof.  \end{proof}

 \begin{theorem}\label{holo}
 Suppose ${\bf X}$ and ${\bf Y}$ are positively omnioriented quasitoric orbifolds and
 $\rho: X \to Y$ a blowdown map as constructed above. Then we may choose torus invariant
 almost complex structures $J_1$ on ${\bf X}$ and $J_2$ on ${\bf Y}$ with respect to
 which $\rho$ is analytic near the exceptional set and almost complex orbifold morphism
  away from the exceptional set.
 \end{theorem}

 \begin{proof}
 We use the notation of theorem \ref{bldn}.
 It is convenient to make the following changes of coordinates. On the
chart $\widetilde{X}_{v_1}$, define
 (compare to equation \eqref{coords}),
\begin{equation}\label{zv1} \begin{array}{lll}
   z_1^{\prime}(v_1) =  z_1(v_1)  \sqrt{\frac{1}{p_2^{k}}},
   & \quad
& z_2^{\prime}(v_1) = z_2(v_1)
\sqrt{\frac{\delta(\widehat{p})^{m}p_2^{m}}{\widehat{p}}}.
   \end{array}
   \end{equation}
This is a valid change of coordinates since
$\frac{\delta(\widehat{p})^{m}}{\widehat{p}}$ and $p_2$ are
nonzero. In these coordinates the map $\rho: X_{v_1} \to Y_w $
takes the form
\begin{equation}\label{holobldn1}  \begin{array}{lll}
   z_1(w) = z_1^{\prime}(v_1) z_2^{\prime}(v_1)^{\frac{k}{m}} ,& \quad
   & z_2(w) = z_2^{\prime}(v_1)^{\frac{1}{m}}.
   \end{array} \end{equation}

Similarly on the chart $\widetilde{X}_{v_2}$ we choose new
coordinates as follows,
    \begin{equation}\label{zv2} \begin{array}{lll}
   z_1^{\prime}(v_2) =  z_1(v_2) \sqrt{\frac{\delta(\widehat{p})^m
   p_1^{\frac{m}{k}} }{\widehat{p}}},
   & \quad & z_2^{\prime}(v_2) = z_2(v_2) \sqrt{p_1^{-\frac{1}{k}}}.
   \end{array}
   \end{equation}

 In these coordinates the map $\rho: X_{v_2} \to Y_w $ takes the form
 \begin{equation}\label{holobldn2}  \begin{array}{lll}
   z_1(w) = z_1^{\prime}(v_2)^{\frac{k}{m}}, & \quad
   & z_2(w) = z_1^{\prime}(v_2)^{\frac{1}{m}}z_2^{\prime}(v_2).
   \end{array} \end{equation}

 Let $U_2$ be a small tubular neighborhood of the edge $E_2$ in $P$.
 Choose complex structures on $\pi^{-1}(U_2) \cap X_{v_1}$ and
  $\pi^{-1}(U_2) \cap X_{v_2}$ by declaring the coordinates
  $z^{\prime}_1(v_1),z^{\prime}_2(v_1) $ and respectively $z^{\prime}_1(v_2),
  z^{\prime}_2(v_2)$ to be holomorphic. These complex structures agree on the
  intersection and define a complex structure $J_1$ on $\pi^{-1}(U_2)$ since
  the coordinates are related as follows.
    \begin{equation}\label{holobldn3} \begin{array}{lll}
   z_1^{\prime}(v_2) =  z_1^{\prime}(v_1)^{\frac{m}{k}} z_2^{\prime}(v_1)
   & \quad & z_2^{\prime}(v_2) = z_1^{\prime}(v_1)^{-\frac{1}{k}}.
   \end{array}
   \end{equation}

   Consider the neighborhood $V:=f(U_2)$ in $\widehat{P}$. On $\pi^{-1}(V) \subset Y_w$
   choose a complex structure $J_2$ by declaring the coordinates $z_1(w), z_2(w)$ to be
   holomorphic. Consequently by equations \eqref{holobldn1} and \eqref{holobldn2}, the
   blowdown map $\rho$ is analytic on $\pi^{-1}(U_2)$.

    We will extend $J_1$ to an almost complex structure on ${\bf X}$. For this purpose
    choose standard metrics  $\mu_j$,
  with respect to the coordinates $(z_1^{\prime}(v_j), z_2^{\prime}(v_j))$,
 on $\pi^{-1}(U_2) \cap X_{v_j}$ for $j=1,2$. Let $\{W_1, W_2\}$ be an open
cover of $U_2$ such that $v_j \in W_j$ and $ W_j^c$ contains a neighborhood
of $v_k$ for $j,k=1,2$ and $j \neq k$.
 Glue $\mu_1$ and $\mu_2$
 by a torus invariant partition of
   unity subordinate to the cover $\{ \pi^{-1}( W_1), \pi^{-1}( W_2)\} $.
   This produces a metric $\mu^{\prime}$ on $\pi^{-1} (U_2)$ such that $J_1$ is orthogonal with respect
    to $\mu^{\prime}$.
 Moreover $\mu^{\prime}$ is standard with respect to the given coordinates near $v_1$ and
 $v_2$.
 Let $U \subset U_2$ be a smaller tubular neighborhood of $E_2$ in $P$.
    Using suitable partition of unity,
     extend $\mu^{\prime}|_U$ to a torus invariant metric $\mu$ on
      ${\bf X}$ such that $\mu $ is standard
      with respect to our choice of coordinates near each vertex.
    Then extend $J_1|_U$ first to the union of the $1$-skeleton of $X$
   and $\pi^{-1}(U)$, and then to entire ${\bf X}$ as a torus invariant orthogonal almost
   complex structure. This can be done in a way similar to the proof of theorem \ref{thmacs}.

 Now  ${\bf X} \bigcap \pi^{-1}(E_2)^c$ is diffeomorphic
 to ${\bf Y} \bigcap \{w\}^c$ via the blowdown map. Thus $J_2^{\prime}:=
  d\rho \circ J_1 \circ d\rho^{-1}$
 is an almost complex structure on ${\bf Y} \bigcap \{w\}^c$. We have the following
 equalities  for a point in the intersection
$({\bf Y} \bigcap \{w\}^c) \bigcap \pi^{-1}(f(U))$.
\begin{equation}
   J_{2} =(d\rho)(d\rho)^{-1}J_{2}=(d\rho)J_1(d\rho)^{-1}=J_{2}^{\prime}
   \end{equation}
    The second equality is due to complex analyticity of $\rho^{-1}$ on
    $ \pi^{-1}(f(U)) \bigcap \{w\}^c$.
    So the two structures are the same and we obtain a torus invariant almost complex
    structure on ${\bf Y}$, which we denote again by $J_2$
  without confusion. The blowdown $\rho$ is an almost complex diffeomorphism
    of orbifolds away from the exceptional set $\pi^{-1}(E_2) \subset {\bf X}$ and the point
    $w \in {\bf Y}$ with respect to $J_1$ and $J_2$. It is an analytic map of complex analytic
    spaces near  $\pi^{-1}(E_2)$ and $w$.
\end{proof}

\begin{defn}\label{smfn} A complex valued continuous function $f: X \to \CC $
 on an orbifold ${\bf X}$ is called smooth if  $f \circ \xi $ is
smooth for any orbifold chart $(\widetilde{U},G,\xi) $. We denote
the sheaf of smooth functions on ${\bf X}$ by $\mathcal{S}_X$.
This is a sheaf on the underlying space $X$ but depends on the
orbifold structure. A smooth function $f$ on an almost complex
orbifold $({\bf X}, J)$ is said to be $J$-holomorphic if the
differential $d(f \circ \xi)$ commutes with $J$ for every chart
$(\widetilde{U},G,\xi)$. We denote the sheaf of $J$-holomorphic
functions on ${\bf X}$ by $\Omega^0_{J,X}$. A continuous map
$\rho: X \to Y $ between almost complex orbifolds $({\bf X}, J_1)$
and $({\bf Y}, J_2) $ is said to be pseudo-holomorphic if $f \circ
\rho \in \Omega^0_{J_1,X}(\rho^{-1}(U))$ for every $f \in
\Omega^0_{J_2,Y}(U) $ for any open set $U  \subset Y$; that is,
$\rho$ pulls back pseudo-holomorphic functions to
pseudo-holomorphic functions.
\end{defn}

\begin{corollary}\label{smbd} The blowdown map $\rho$ of theorem \ref{bldn} is
 pseudo-holomorphic with respect to the almost complex structures given in
theorem \ref{holo}.
\end{corollary}

\begin{proof} Since $\rho$ is a an almost complex diffeomorphism of orbifolds
away from the exceptional sphere $\pi^{-1}(E_2)$, it suffices to
check the statement on the open sets $X_{v_1}$ and $X_{v_2}$. The
ring $\Omega^0_{J_1, X}(X_{v_1})$ is the ring of convergent power
series in variables $z_i^{\prime}(v_1)$ as $v_1$ is a smooth
point. The ring $\Omega^0_{J_2, Y}(Y_{w})$ is the
$\ZZ_m$-invariant subring of convergent power series in variables
$z_i^{\prime}(w)$. Let $\eta $ denote a primitive $m$-th root of
unity. Then the action of $\ZZ_m$ on $\widetilde{Y}_w$ is given by
$\eta \cdot(z_1(w), z_2(w)) = (\eta^k z_1(w), \eta z_2(w)) $, see
\eqref{action}. Therefore the invariant subring  is generated by
$\{z_1(w)^{i} z_2(w)^{j}: m|(ik + j) \}$. Using \eqref{holobldn1},
we have
\begin{equation}\label{rhostar}
z_1(w)^{i} z_2(w)^{j} \circ \rho = z_1^{\prime}(v_1)^i  z_2^{\prime}(v_1)^{\frac{ik+j}{m}}.
\end{equation}
When $m$ divides $(ik+j)$,  $z_1(w)^{i} z_2(w)^{j} \circ \rho$ belongs to $\Omega^0_{J_1, X}(X_{v_1})$.

Similarly the $\ZZ_k$ action on $\widetilde{X}_{v_2}$ is given by
$\zeta \cdot (z_1^{\prime}(v_2), z_2^{\prime}(v_2)  ) =
(\zeta^{\frac{m}{k}} z_1^{\prime}(v_2), \zeta^{\frac{-1}{k}}  z_2^{\prime}(v_2)  )$ where
 $\zeta$ is a primitive $k$-th root of unity. The ring
 $\Omega^0_{J_1, X}(X_{v_2})$ is generated by $\{z^{\prime}_1(v_2)^{a}
z^{\prime}_2(v_2)^{b}: k|(am - b) \}$. Using \eqref{holobldn2} we
have,
\begin{equation}
z_1(w)^{i} z_2(w)^{j} \circ \rho = z_1^{\prime}(v_2)^{\frac{ik +j}{m}}
    z_2^{\prime}(v_2)^j.
\end{equation}
Since $ \frac{ik +j}{m} m - j = ik   $ is divisible by $k$,
$z_1(w)^{i} z_2(w)^{j} \circ \rho$ belongs to $\Omega^0_{J_1,
X}(X_{v_2})$ if $z_1(w)^{i} z_2(w)^{j} \in \Omega^0_{J_2,
Y}(Y_{w}) $.
\end{proof}

\begin{remark} The blowdown map $\rho$ does not pull back a smooth
function to a smooth function in general.
\end{remark}

\section{McKay correspondence}

\begin{defn} Given an almost complex $2n$-dimensional orbifold $({\bf X}, J)$, we define the
 canonical sheaf $K_{X}$ to be the sheaf of continuous $(n,0)$-forms on $X$;
  that is, for any orbifold chart
$(\widetilde{U},G,\xi)$ over an open set $U \subset X$,
 $K_X (U) = \Gamma( \wedge^n \mathcal{T}^{1,0}(\widetilde{U})^{\ast})^G$ where $\Gamma$ is the functor
 that takes continuous sections.
\end{defn}

An orbifold singularity $\CC^n /G$ is said to be $SL$ if $G$ is a
finite subgroup of  $SL(n, \CC)$. For an $SL$-orbifold ${\bf X}$,
i.e. one whose singularities are all $SL$, the canonical sheaf is
a complex line bundle over $X$.

\begin{defn}\label{defcrep}
 A pseudo-holomorphic blowdown map $\rho:{\bf X}\to {\bf Y}$ between two four
 dimensional primitive positively omnioriented quasitoric $SL$ orbifolds is said to be crepant if
 $\rho^{\ast} K_Y = K_X$.
\end{defn}

\begin{lemma}\label{SL} The orbifold singularities corresponding to the characteristic vectors
$\lambda_1 =(1,0), \, \lambda_2 = (-k,m) $ and
$\lambda_1 =(0,1), \, \lambda_2 = (-k,m) $ are $SL$ if $k+1 =m$.
\end{lemma}

\begin{proof} Consider the case $\lambda_1 =(1,0), \, \lambda_2 = (1- m,m) $.
 Refer to the description of the action of the local group in equation \eqref{action}.
Any element of this group is represented by an integral vector $a_1 \lambda_1 +
a_2 \lambda_2$. Integrality of such a vector implies that $a_1  + a_2 - m a_2$ and
$m a_2$ are integers. Hence $a_1 +a_2$ is also an integer. This implies that the group
acts as a subgroup of  $SL(2, \CC)$.
The other case is similar.
\end{proof}

\begin{lemma}\label{crepant} Suppose ${\bf X}$ and ${\bf Y}$ are two $4$-dimensional primitive, positively
omnioriented, quasitoric $SL$ orbifolds and  $\rho: {\bf X}\to
{\bf Y}$ is a pseudo-holomorphic  blowdown as constructed in
theorems \ref{bldn} and \ref{holo}. Then $\rho$  is crepant if and
only if $ k+1=m$.
\end{lemma}

\begin{proof} We consider the canonical sheaf as a sheaf of modules
 over the sheaf of continuous functions $\mathcal{C}^0_X$.
  Since $\rho$ is an almost complex diffeomorphism away
from the exceptional set it suffices to check the equality of the
$\rho^{\ast} K_Y$ and $K_X$ on the neighborhood $\pi^{-1} (U)
\subset X$ of the exceptional set, defined in proof of theorem
\ref{holo}.

On $X_{v_1} \bigcap \pi^{-1} (U)  $, the sheaf $K_X$ is generated
over the sheaf $\mathcal{C}^0_{X}$ by the form $dz_1^{\prime}
(v_1)\wedge dz_2^{\prime}(v_1)$. On the other hand
$\rho^{\ast}(K_Y)$ is generated on this neighborhood over
$\mathcal{C}^0_{X}$ by
\begin{equation} \begin{array}{l}
 \rho^{\ast} dz_1(w)\wedge dz_2(w) \\
= d(z_1^{\prime}(v_1) z_2^{\prime}(v_1) ^{\frac{k}{m}} ) \wedge d(z_2^{\prime}(v_1)^{\frac{1}{m}})\\
= z_2^{\prime}(v_1) ^{\frac{k}{m}} dz_1^{\prime}(v_1) \wedge
\frac{1}{m}(z^{\prime}_2(v_1))^{\frac{1}{m}-1}
dz_2^{\prime}(v_1)\\
=\frac{1}{m} (z^{\prime}_2)^{\frac{k+1}{m}-1} dz_1^{\prime} (v_1)\wedge dz_2^{\prime}(v_1).
\end{array}
\end{equation}

Thus $\rho^{\ast} K_Y = K_X$ on $X_{v_1} \bigcap \pi^{-1} (U)  $ if and only if $k+1 =m$. Similarly on
 $X_{v_2} \bigcap \pi^{-1} (U)  $, the sheaf $K_X$ is generated over $\mathcal{C}^0_X$
by the form $dz_1^{\prime} (v_2)\wedge dz_2^{\prime}(v_2)$. On the
other hand $\rho^{\ast}(K_Y)$ is generated on this neighborhood
over $\mathcal{C}^0_{X}$ by
\begin{equation} \begin{array}{l}
\rho^{\ast} dz_1(w)\wedge dz_2(w) \\
= d (z_1^{\prime}(v_2)^{\frac{k}{m}}) \wedge
  d ( z_1^{\prime}(v_2)^{\frac{1}{m}}z_2^{\prime}(v_2) )\\
= \frac{k}{m} (z_1^{\prime}(v_2))^{\frac{k}{m} -1} dz_1^{\prime}(v_2)
\wedge  z_1^{\prime}(v_2)^{\frac{1}{m}} dz_2^{\prime}(v_2)\\
= \frac{k}{m} (z_1^{\prime}(v_2))^{\frac{k+1}{m} -1} dz_1^{\prime} (v_2)\wedge dz_2^{\prime}(v_2).
\end{array}
\end{equation}

Again $\rho^{\ast} K_Y = K_X$ on $X_{v_2} \bigcap \pi^{-1} (U)  $ if and only if $k+1 =m$ .
\end{proof}

It is easy to construct examples of a positively omnioriented quasitoric $SL$ manifold or orbifold
 which is not a toric variety and admits a crepant blowdown. For instance, let
${\bf X}$  be the quasitoric manifold  over a $7$-gon with characteristic
vectors $(1,0)$, $ (0,1)$, $(-1,2)$, $(-2,3)$, $(1,-2)$, $(0,1)$ and $(-1,-1)$.  Then ${\bf X}$ is
positively omnioriented, primitive and $SL$. However it is not a toric variety as its Todd genus is $2$.
The orbifold ${\bf Y}$ over a $6$-gon with characteristic vectors $(1,0)$,  $(-1,2)$, $(-2,3)$,
 $(1,-2)$, $(0,1)$ and $(-1,-1)$ is a crepant blowdown of ${\bf X}$. Same holds for the orbifold
${\bf Z}$ over a $6$-gon with characteristic
vectors $(1,0)$, $ (0,1)$, $(-2,3)$, $(1,-2)$, $(0,1)$ and $(-1,-1)$. It may be argued that ${\bf Y}$
and ${\bf Z}$ are not toric varieties as otherwise the blowup ${\bf X}$ would be a toric variety.

The singular and Chen-Ruan cohomology  groups (see \cite{[CR]}) of an almost complex quasitoric
 orbifold was calculated in \cite{[PS]}. For a four dimensional  primitive positively omnioriented
 quasitoric orbifold
${\bf X}$, the Chen-Ruan cohomology groups are
\begin{equation}\label{crc}
H^{d}_{CR}({\bf X}, \QQ) = \left\{ \begin{array}{ll}
H^{0}(X,\QQ) & {\rm if \; d=0}\\
 &\\
 H^{d}(X,\QQ) \oplus \bigoplus_{2{\rm age}(g)=d} \QQ(v,g) & {\rm if\; d>0}
\end{array} \right.
\end{equation}
Here $v$ varies over vertices of $P$, $g$ varies over nontrivial elements of the local group $G_v$ and
${\rm age}(g)$ is the degree shifting number $\iota(g)$ of \cite{[CR]}.

The singular cohomology groups of $X$ are
\begin{equation}\label{sc}
H^{d}( X, \QQ) \cong \left\{ \begin{array}{ll}
\QQ & {\rm if \; d=0,4}\\
 \bigoplus_{m-2} \QQ & {\rm if\; d=2}\\
 0 & {\rm otherwise}
\end{array} \right.
\end{equation}
where $m$ denotes the number of edges of $P$.

\begin{theorem} Suppose $\rho:{\bf X} \to {\bf Y}$ be a
 pseudo-holomorphic blowdown between four dimensional
primitive positively omnioriented quasitoric manifolds,
 as constructed in theorems \ref{bldn} and \ref{holo}. If $\rho$ is crepant then
 $$  {\rm dim} \left( H^{d}_{CR}({\bf X}, \QQ) \right) =
  {\rm dim} \left(  H^{d}_{CR}({\bf Y}, \QQ) \right).$$
\end{theorem}

\begin{proof} By formulas \eqref{crc} and \eqref{sc}, it
  suffices to compare the contributions of the edge $E_2$ and vertices $v_1,\, v_2$ of $P$
to $H^{d}_{CR}({\bf X}, \QQ) $ with the contribution of the vertex $w$ of $\widehat{P}$ to $H^{d}_{CR}({\bf Y}, \QQ) $.

The edge $E_2$ contributes a single generator to $H^{2}_{CR}({\bf X}, \QQ) $.
The vertex $v_1$  has no contribution as it is a smooth point and $G_{v_1}$ is trivial.
The group $G_{v_2}$ is isomorphic to $\ZZ_k$. Assume $\rho$ is crepant. Then by
lemma \ref{crepant} $m = k+1$ and the characteristic vector $\lambda_3 =(-k,k+1)$.
Recall that $\lambda_2=(0,1)$.
 The elements of $G_{v_2}$ are
\begin{equation}
g^p := (-p,0) = \frac{-(k+1)p}{k}  \lambda_2 + \frac{p}{k}\lambda_3 \;
{\rm where} \;
0\le p \le k-1.
\end{equation}
By equations \eqref{action} and \eqref{zv2} the action of $g^p$ is given by
\begin{equation}
g^p \left( z_1^{\prime}(v_2), \, z_2^{\prime}(v_2) \right)
  =\left(e^{2\pi\sqrt{-1} (-p- \frac{p}{k})} z_1^{\prime}(v_2), \;
  e^{2\pi\sqrt{-1} \frac{p}{k}} z_2^{\prime}(v_2)\right).
    \end{equation}
   The degree shifting number
   \begin{equation}
   {\rm age}(g^p) = (1-\frac{p}{k}) + \frac{p}{k} = 1.
   \end{equation}
   Thus each $g^p$, $1\le p \le k-1$, contributes a generator $(v_2,g^p)$ to
   $H^{2}_{CR}({\bf X}, \QQ) $.

   The characteristic vectors at $w$ are $\lambda_1=(1,0)$ and
   $\lambda_3=(-k,k+1)$. The group $G_w \cong \ZZ_{k+1}$ has elements
   \begin{equation}
h^q := (0,q) = \frac{qk}{k+1}  \lambda_1 + \frac{q}{k+1}\lambda_3 \;
{\rm where} \;
0\le q \le k.
\end{equation}
   By equations \eqref{action} the action of $h^q$ is given by
\begin{equation}
h^q \left( z_1(w), \, z_2(w) \right)
  =\left(e^{2\pi\sqrt{-1} (q- \frac{q}{k+1})} z_1(w), \;
  e^{2\pi\sqrt{-1} \frac{q}{k+1}} z_2(w)\right).
    \end{equation}
   The degree shifting number
   \begin{equation}
   {\rm age}(h^q) = (1-\frac{q}{k+1}) + \frac{q}{k+1} = 1.
   \end{equation}
   Thus each $h^q$, $1\le q \le k$, contributes a generator $(w,h^q)$ to
   $H^{2}_{CR}({\bf Y}, \QQ) $.

   Hence the contributions of the edge $E_2$ and vertices $v_1,\, v_2$ of $P$
to dimension of $H^{d}_{CR}({\bf X}, \QQ) $ match the contribution of the vertex
$w$ of $\widehat{P}$ to dimension of
$H^{d}_{CR}({\bf Y}, \QQ) $.

\end{proof}

{\bf ACKNOWLEDGEMENT.} It is a pleasure to thank Indranil Biswas, Mahuya Datta,
 Goutam Mukherjee, B. Doug Park and Soumen Sarkar for many useful
 conversations. We also thank Mikiya Masuda and an anonymous referee for pointing out some errors
  in earlier drafts. We thank NBHM and Universidad de los Andes for providing financial
  support for our research activities.

\renewcommand{\refname}{References}

\vspace{2cm}

\end{document}